\documentclass[11pt,reqno]{amsart}

\usepackage[utf8]{inputenc}
\usepackage{amsmath,amssymb}
\usepackage{graphicx}
\usepackage[margin=1in]{geometry}
\usepackage{needspace}
\usepackage{booktabs}
\usepackage{microtype}
\usepackage[hidelinks,pdfencoding=auto,psdextra]{hyperref}

\newcommand{\R}{\mathbb{R}}
\newcommand{\N}{\mathbb{N}}
\newcommand{\Nzero}{\mathbb{N}_0}
\newcommand{\PP}{\mathbb{P}}
\newcommand{\EE}{\mathbb{E}}
\newcommand{\one}{\mathbf{1}}
\DeclareMathOperator{\Leb}{Leb}

\numberwithin{equation}{section}
\allowdisplaybreaks
\theoremstyle{plain}
\newtheorem{theorem}{Theorem}[section]
\newtheorem{lemma}[theorem]{Lemma}
\newtheorem{proposition}[theorem]{Proposition}
\newtheorem{corollary}[theorem]{Corollary}

\theoremstyle{remark}
\newtheorem{remark}[theorem]{Remark}

\author[A. Guha Roy]{Aditya Guha Roy}
\address{Department of Mathematics, Indiana University, Bloomington, Indiana 47405, USA}
\email{guharoyagraditya@gmail.com}

\author[Y. Peres]{Yuval Peres}
\address{Beijing Institute of Mathematical Sciences and Applications}
\email{yperes@gmail.com}

\author[S. Qin]{Shuo Qin}
\address{Beijing Institute of Mathematical Sciences and Applications, and Yau Mathematical Sciences Center, Tsinghua University}
\email{qinshuo@bimsa.cn}

\author[J. Zuo]{Junchi Zuo}
\address{Qiuzhen College, Tsinghua University}
\email{zuojc21@mails.tsinghua.edu.cn}

\title[Survival in reactive gambling]{Survival probability in reactive gambling: monotonicity, singularity, and H\"older continuity}
\date{}


\begin{document}

\begin{abstract}
A gambler starts with fortune $x$ and unit bet, then aims to double  the bet after every win and halve it after every loss, while the actual bet is capped by the current fortune. Each round is won independently with probability $p\in(0,1)$. We study $f(x,p)$, the probability that the gambler is never ruined. We prove that $f(x,p)>0$ exactly when $x>2$ and $p<1/2$. For fixed $x>2$, $f(x,\cdot)$ is strictly decreasing and real analytic on $(0,1/2)$, and vanishes on $[1/2,1)$: playing a more favourable game lowers the chance of surviving it. For fixed $p\in(0,1/2)$, $f(\cdot,p)$ is strictly increasing and singular continuous on $[2,\infty)$. An elementary two-step argument gives an explicit global H\"older exponent, with constant one, which is strictly larger than $\log_2(1+p)$ and tends to $\frac12\log_2(111/46)=0.635426\ldots$ as $p\uparrow1/2$; we place it in a hierarchy of computable exponents and show, by an exact finite computation, that the optimal exponent is strictly smaller than the pointwise exponent at $x=2$ once $p$ is close to $1/2$. We determine that boundary exponent, and give the complete phase diagram when wins and losses use unequal multipliers.
\end{abstract}

\maketitle

\section{Introduction}
\label{sec:introduction}

 Xu and Harvey~\cite[Section~3.5]{XuHarvey2014} report
larger bets after wins and smaller bets after losses in online sports
betting. We study a gambler whose target bet doubles after a win and halves after
a loss, with each actual bet capped by the current fortune.    Figure~\ref{fig:survival} illustrates the resulting
survival probability.

\begin{figure}[!h]
  \centering
  \includegraphics[width=0.70\textwidth]{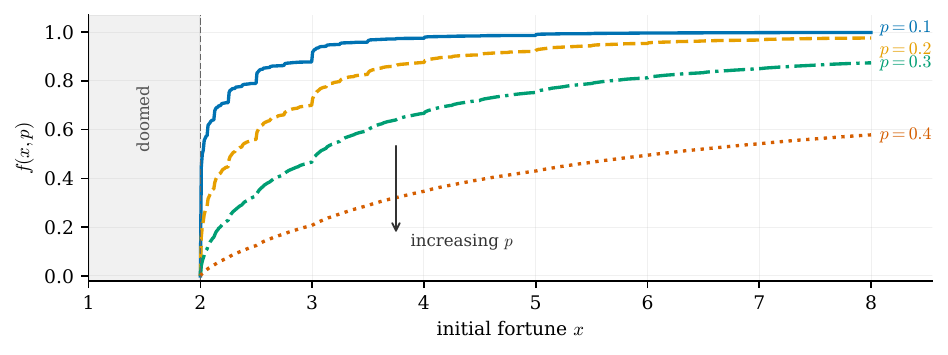}
  \caption{The survival probability $f(x,p)$: the chance that a gambler who
  starts with fortune $x$ and follows \eqref{eq:capped-model} is never ruined.
  Ruin is certain for $x\le2$ (Theorem~\ref{thm:phase}); above this
  threshold the curves are ordered \emph{downward} in
  $p$, so that winning more often makes ruin more likely
  (Corollary~\ref{cor:strict-p}). On $[2,\infty)$ each is continuous and
  strictly increasing, with derivative zero almost everywhere
  (Theorem~\ref{thm:regularity}). The curves are Monte Carlo approximations based on
  truncated versions of \eqref{eq:survival-representation-intro}.}
  \label{fig:survival}
\end{figure}

Two features of that picture are what this paper is about: the curves are
ordered downward in the winning probability, so that a gambler who wins more
often survives less often; and each of them, although continuous and strictly
increasing in the initial fortune, has derivative zero at almost every point.
Idealize the rule as follows. Let $W_0=x\ge1$ be the initial fortune and let
$B_1=1$ be the first target bet. Let $(\xi_n)_{n\ge1}$ be independent random variables with
\[
  \PP(\xi_n=1)=p,
  \qquad
  \PP(\xi_n=-1)=1-p,
\]
where $\xi_n=1$ represents a win in round $n$ and $\xi_n=-1$ represents a loss. Before ruin, define the actual bet and the subsequent target bet by
\begin{equation}
\label{eq:capped-model}
  C_n:=\min\{B_n,W_{n-1}\},
  \qquad
  W_n:=W_{n-1}+\xi_n C_n,
  \qquad
  B_{n+1}:=B_n2^{\xi_n}.
\end{equation}
Thus the target bet doubles after a win and halves after a loss, while the gambler never wagers more than the current fortune. The target bet continues to follow the double-or-halve rule when the cap is active; its value after ruin is immaterial. Let
\[
  \tau:=\inf\{n\ge0:W_n=0\},
  \qquad
  f(x,p):=\PP_{x,p}(\tau=\infty).
\]
The process is absorbed at zero after time $\tau$.

The endpoint cases are deterministic. When $p=0$, one has $f(x,0)=0$ for $x<2$ and $f(x,0)=1$ for $x\ge2$. In particular, from $x=2$ the fortunes $2,1,1/2,\ldots$ tend to zero without reaching it in finite time. When $p=1$, every round is won and $f(x,1)=1$. We therefore assume $0<p<1$ unless an endpoint is mentioned explicitly. These endpoints require separate treatment when strong survival is interpreted through the shadow chain, because the usual implication ``entry into the doomed region eventually causes ruin'' uses the presence of both wins and losses.

Our first result gives the phase diagram.

\begin{theorem}
\label{thm:phase}
For $x\ge1$ and $0<p<1$, the gambler survives with positive probability, equivalently $f(x,p)>0$, if and only if
\[
  x>2
  \qquad\text{and}\qquad
  p<\frac12.
\]
In this regime, $0<f(x,p)<1$.
\end{theorem}

The cap in \eqref{eq:capped-model} changes the wealth path after it becomes active. It does not, however, change the bankruptcy time: for $0<p<1$, the capped and uncapped shadow games go broke at the same round almost surely. Section~\ref{sec:phase} proves this by an explicit coupling. The strong-survival probability $f_n$ introduced below is a different finite-horizon quantity: it records whether the shadow chain has already entered the region from which ruin is inevitable.

The same coupling yields a perpetuity representation. Put $R_0=0$, $R_n=\sum_{i=1}^n\xi_i$, and
\begin{equation}
\label{eq:perpetuity-intro}
  \mathcal S:=\sum_{n=1}^{\infty}\one_{\{\xi_n=1\}}2^{R_{n-1}},
\end{equation}
with the value $+\infty$ allowed. For $x>2$,
\begin{equation}
\label{eq:survival-representation-intro}
  f(x,p)=\PP(\mathcal S\le x-2).
\end{equation}
If $\mathcal S'$ is an independent copy, then
\begin{equation}
\label{eq:fixed-point-intro}
  \mathcal S\ \stackrel{\mathrm d}{=}\
  \begin{cases}
    1+2\mathcal S',&\text{with probability }p,\\
    \mathcal S'/2,&\text{with probability }1-p.
  \end{cases}
\end{equation}
This places the convergence threshold in the classical theory of random affine recursions and perpetuities \cite{Vervaat1979,GoldieMaller2000,DiaconisFreedman1999,AlsmeyerIksanovRosler2009}: the multiplicative coefficient $D$ satisfies
\[
  \EE\log D=(2p-1)\log2.
\]
For $p<1/2$, the series $\mathcal S$ converges almost surely, and
$f(x,p)=\PP(\mathcal S\le x-2)$ expresses the survival probability in
terms of its distribution function. Our focus is this two-parameter
survival probability, especially its counterintuitive monotonicity and the
contrast between its regularity in $x$ and in $p$. Functional equations
generating singular distribution functions have a long history; see, for
example, \cite{Kinney1958,Salem1943,Kairies1997}.

For $0<p<1/2$, qualitative H\"older continuity follows from the regularity theorem of Aoun and Guivarc'h \cite[Theorem~2.12]{AounGuivarch2020}, in the affine formulation recorded by B\'enard, He and Zhang \cite[Lemma~2.1(ii)]{BenardHeZhang2026}. Indeed, the law of $\mathcal S$ is stationary for
\[
  g_0(s)=s/2,
  \qquad
  g_1(s)=1+2s,
\]
chosen with probabilities $1-p$ and $p$. The driving law has finite support, the two maps have no common fixed point, and their mean logarithmic dilation is $(2p-1)\log2<0$. The cited results therefore give constants $C,\gamma>0$ such that
\[
  \sup_{u\in\R}\PP\bigl(\mathcal S\in u+[-r,r]\bigr)
  \le Cr^\gamma,
  \qquad r>0.
\]
For comparison, Feng and Lau \cite[Proposition~2.2]{FengLau2009} proved an elementary Frostman estimate for finite systems of contracting similarities; that result does not directly cover the expansive map $g_1$. Our proof is elementary and supplies an explicit exponent and constant.

The normalized variable $\widehat Y_n$ in
\eqref{eq:normalized-shadow-ratio}, whose dynamics are given by
\eqref{eq:Y-recursion} and whose finite-horizon survival event is recorded
in \eqref{eq:strong-survival-definition}, has a simple interpretation: it
is shadow wealth measured in units of the next target bet. In the safe
region, a win decreases this normalized safety margin, whereas a loss
increases it. This is the mechanism behind the following result.

\begin{corollary}
\label{cor:strict-p}
For every fixed $x>2$, the function $p\mapsto f(x,p)$ is strictly decreasing on $(0,1/2)$.
\end{corollary}

For finite-horizon approximation and the functional equation, let $\Nzero:=\{0,1,2,\ldots\}$ and introduce the Markov operator
\begin{equation}
\label{eq:markov-operator}
  (\mathcal L_pu)(x)
  :=p\,u\!\left(\frac{x+1}{2}\right)
    +(1-p)u(2x-2).
\end{equation}
Set
\[
  f_0(x,p):=\one_{(2,\infty)}(x),
\]
and, for $n\in\Nzero$, define
\begin{equation}
\label{eq:finite-recursion}
  f_{n+1}(x,p)
  =p f_n\!\left(\frac{x+1}{2},p\right)
   +(1-p)f_n(2x-2,p).
\end{equation}
Equivalently, $f_n=\mathcal L_p^nf_0$. Section~\ref{sec:phase} shows that $f_n(x,p)$ is the probability that the normalized shadow chain stays strictly above $2$ at times $0,\ldots,n$. We call it the \emph{$n$-step strong-survival probability}. This is stronger than merely being solvent after $n$ rounds in the capped game: entry into $(-\infty,2]$ makes eventual ruin certain when $0<p<1$, but need not cause immediate bankruptcy.

\begin{proposition}
\label{prop:approximants}
For $0<p<1$, the strong-survival probabilities decrease pointwise to the eventual survival probability:
\[
  f_n(x,p)\downarrow f(x,p).
\]
In particular,
\begin{equation}
\label{eq:self-similarity}
  f(x,p)
  =p f\!\left(\frac{x+1}{2},p\right)
   +(1-p)f(2x-2,p),
  \qquad
  (x,p)\in(2,\infty)\times(0,1/2).
\end{equation}
For $p\in(0,1/2)$, extend $f(\cdot,p)$ to $\R$ by zero on $(-\infty,2]$. With this extension, \eqref{eq:self-similarity} holds for every $x\in\R$, and the boundary values are
\[
  f(x,p)=0\quad(x\le2),
  \qquad
  \lim_{x\to\infty}f(x,p)=1.
\]
For each $p\in(0,1/2)$, moreover, $f(\cdot,p)$ is the maximal $[0,1]$-valued $\mathcal L_p$-harmonic function that vanishes on $(-\infty,2]$.
\end{proposition}

The boundary structure becomes especially transparent after translation. Define
\begin{equation}
\label{eq:translated-survival}
  F_p(t):=f(t+2,p),
  \qquad t\in\R,
\end{equation}
so that $F_p(t)=0$ for $t\le0$. Equation~\eqref{eq:self-similarity} becomes the two-scale equation
\begin{equation}
\label{eq:translated-functional-equation}
  F_p(t)
  =pF_p\!\left(\frac{t-1}{2}\right)
   +(1-p)F_p(2t).
\end{equation}
In particular, $F_p(t)=(1-p)F_p(2t)$ for $0<t\le1$. This exact scaling relation drives the boundary exponent in Theorem~\ref{thm:holder}.

Each finite approximant is a nondecreasing, left-continuous step function in $x$ and a polynomial in $p$; it is also nonincreasing in both $n$ and $p$. See Figure~\ref{fig:approximants}.

For a fixed outcome sequence, the wealth update in \eqref{eq:capped-model} is nondecreasing in the initial fortune. Hence $f(\cdot,p)$ is nondecreasing. Strict increase and quantitative regularity are subtler. Gambling theory supplies other singular functions: Dubins and Savage developed bold play for reaching a fixed target \cite{DubinsSavage1965}, and Billingsley and Neidinger studied singular functions associated with bold play \cite{Billingsley1983,Neidinger2016}. The present state-dependent target-bet model is different, but produces another singular gambling function.

\Needspace{18\baselineskip}
\begin{theorem}
\label{thm:holder}
Fix $p\in(0,1/2)$ and write $q:=1-p$ and
\begin{equation}
\label{eq:explicit-holder-exponent}
  c:=\frac{pq}{1+p},
  \qquad
  \rho(p):=\frac{1+q^2/(1+c^2)}{2(1+p)},
  \qquad
  \alpha(p):=-\frac12\log_2\rho(p).
\end{equation}
Then, for every $x\ge2$ and $r>0$,
\begin{equation}
\label{eq:holder-bound}
  0<f(x+r,p)-f(x,p)
  \le r^{\alpha(p)}.
\end{equation}
The exponent satisfies
\begin{equation}
\label{eq:holder-exponent-comparison}
  \alpha(p)>\log_2(1+p),
  \qquad
  \lim_{p\uparrow1/2}\alpha(p)
  =\frac12\log_2\frac{111}{46}
  =0.635426\ldots.
\end{equation}
Moreover,
\begin{equation}
\label{eq:boundary-exponent}
  \lim_{x\downarrow2}
  \frac{\log f(x,p)}{\log(x-2)}
  =\log_{1/2}(1-p).
\end{equation}
\end{theorem}

\begin{figure}[htbp]
  \centering
  \includegraphics[width=\textwidth]{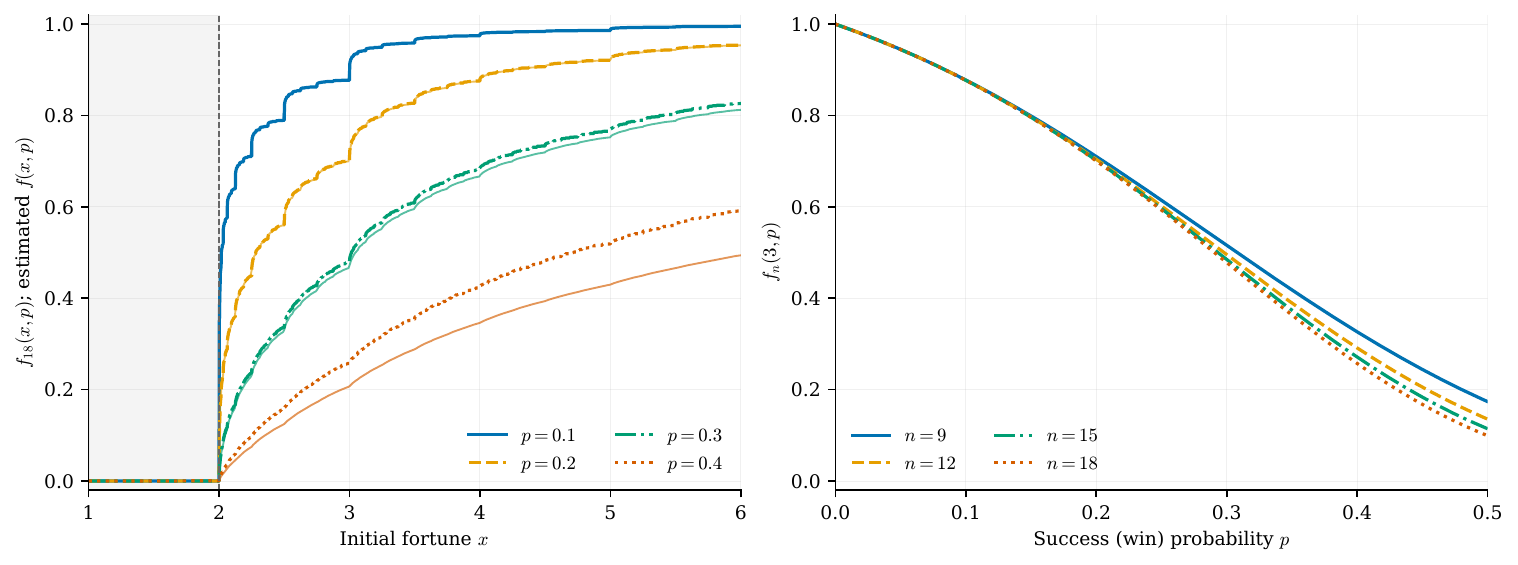}
  \caption{Finite-horizon strong-survival probabilities.
  Left: the step functions $x\mapsto f_{18}(x,p)$ for $p=0.1,0.2,0.3,0.4$,
  obtained by enumerating all outcome words in
  \eqref{eq:finite-prefix-representation}, with Monte Carlo approximations
  to $f(\cdot,p)$ shown as thin curves of the same colour. The latter use
  $6\times10^5$ samples of a truncated series. Both functions vanish for
  $x\le2$. Right: the polynomials $p\mapsto f_n(3,p)$ for $n=9,12,15,18$.
  The approximants decrease to $f$ as $n\to\infty$.}
  \label{fig:approximants}
\end{figure}

The boundary exponent shows that no global H\"older exponent can exceed
$-\log_2(1-p)$. In particular, $f(\cdot,p)$ is not Lipschitz at $2$.

\begin{theorem}
\label{thm:regularity}
\leavevmode
\begin{enumerate}
\item For $p\in(0,1/2)$, the function $f(\cdot,p)$ is strictly increasing and singular continuous on $[2,\infty)$: it is continuous and its derivative is zero for Lebesgue almost every $x$. It extends constantly as $0$ on $[1,2]$ and tends to $1$ as $x\to\infty$.
\item For $x>2$, the function $f(x,\cdot)$ is real analytic on $(0,1/2)$.
\end{enumerate}
\end{theorem}

\subsection{Related work}
\label{sec:related}

Betting systems are surveyed by Ethier~\cite[Ch.~8]{Ethier2010}.
Rigorous analyses include Ethier's study of Oscar's
system~\cite{Ethier1996}, whose bet increases additively after wins,
and Han and Wang's estimates for the largest bet in the Labouchere
system~\cite{HanWang2019}. The double-or-halve rule studied here uses
multiplicative adjustments in both directions.

The qualitative H\"older continuity supplied by
\cite{AounGuivarch2020,BenardHeZhang2026} was explained above.
Gorodetski, Kleptsyn and Monakov~\cite{GorodetskiKleptsynMonakov2026}
prove a broader regularity theorem for random dynamical systems on compact
manifolds, under a moment condition and the absence of a common invariant
probability measure. Their theorem does not directly apply to the
projective compactification of our affine system, since both maps fix
infinity. In the functional-equations literature, Pan~\cite{Pan1994}
derives H\"older continuity of generalized de Rham functions from dilation
equations. Our contribution is an elementary bound with an explicit
exponent and constant, together with the concentration estimates and
coefficient criterion in Section~\ref{sec:refined}.

The coefficient formula in Section~\ref{sec:refined} is a weighted sum
over binary partitions. The corresponding unweighted counting problem
was studied by Mahler~\cite{Mahler1940}, de Bruijn~\cite{deBruijn1948}
and Knuth~\cite{Knuth1966}. Their work provides context for the dyadic
generating functions that arise here.

For a different example of nonmonotonic survival behaviour, Bertacchi
and Zucca~\cite{BertacchiZucca2014} show that strong local survival of
a branching random walk need not be monotone in the reproduction rate.
The mechanism in Corollary~\ref{cor:strict-p} instead comes from the
dependence of future bets on past outcomes.

\section{The capped game, the shadow chain, and the phase transition}
\label{sec:phase}

The target bets are determined by the outcomes:
\begin{equation}
\label{eq:target-bet}
  B_n=2^{R_{n-1}},
  \qquad n\ge1.
\end{equation}
Introduce the uncapped shadow wealth
\[
  \widehat W_0=x,
  \qquad
  \widehat W_n=\widehat W_{n-1}+\xi_nB_n,
\]
and its normalized version
\begin{equation}
\label{eq:normalized-shadow-ratio}
  \widehat Y_n:=\frac{\widehat W_n}{B_{n+1}}.
\end{equation}
The ruin time of the shadow chain is
\[
  \widehat\tau:=\inf\{n\ge0:\widehat W_n\le0\}.
\]
The shadow chain obeys
\begin{equation}
\label{eq:Y-recursion}
  \widehat Y_{n+1}=
  \begin{cases}
    (\widehat Y_n+1)/2,&\xi_{n+1}=1,\\
    2(\widehat Y_n-1),&\xi_{n+1}=-1.
  \end{cases}
\end{equation}

The cap acts only when normalized wealth is at most $1$. Indeed, while the capped process is alive, $Y_n:=W_n/B_{n+1}$ has transition
\begin{equation}
\label{eq:capped-Y-recursion}
  Y_{n+1}=
  \begin{cases}
    (Y_n+1)/2,&Y_n>1,\ \xi_{n+1}=1,\\
    2(Y_n-1),&Y_n>1,\ \xi_{n+1}=-1,\\
    Y_n,&0<Y_n\le1,\ \xi_{n+1}=1,\\
    0,&0<Y_n\le1,\ \xi_{n+1}=-1.
  \end{cases}
\end{equation}

For $a\in\R$, let
\[
  \sigma_a:=\inf\{n\ge0:\widehat Y_n\le a\}.
\]
Because $(-\infty,2]$ is forward invariant under both maps in \eqref{eq:Y-recursion}, the recursion \eqref{eq:finite-recursion} gives
\begin{equation}
\label{eq:strong-survival-definition}
\begin{aligned}
  f_n(x,p)
  &=\PP_{x,p}(\sigma_2>n)\\
  &=\PP_{x,p}(\widehat Y_j>2\text{ for every }0\le j\le n)
   =\PP_{x,p}(\widehat Y_n>2).
\end{aligned}
\end{equation}
Thus strong survival is stronger than capped-game solvency through time $n$. For example, from $x=3$ a first win sends $\widehat Y_1$ to $2$, so $f_1(3,p)=1-p$, although $\PP_{3,p}(\tau>1)=1$.

\begin{lemma}
\label{lem:cap-equivalence}
For $0<p<1$, the capped and shadow bankruptcy times satisfy
\[
  \tau=\widehat\tau\qquad\text{almost surely}.
\]
Consequently,
\begin{equation}
\label{eq:cap-equivalence}
  f(x,p)=\PP_{x,p}(\tau=\infty)
  =\PP_{x,p}(\widehat\tau=\infty)
  =\PP_{x,p}(\sigma_1=\infty)
  =\PP_{x,p}(\sigma_2=\infty).
\end{equation}
\end{lemma}

\begin{proof}
Put $N:=\sigma_1$. The capped and shadow processes agree up to $N$.
On $\{N=\infty\}$ neither process goes broke. On $\{N<\infty\}$,
write $y:=\widehat Y_N\in(0,1]$ and let
\[
  L:=\inf\{n>N:\xi_n=-1\}.
\]
Here $y>0$ because a first entrance from above $1$ has positive ratio;
the case $N=0$ can only have $x=1$. Since $p<1$, the time $L$ is
finite almost surely. During the wins at rounds $N+1,\ldots,L-1$, the
capped normalized ratio remains $y$, whereas after $j$ such wins the
shadow ratio is
\[
  1-2^{-j}(1-y)\in(0,1].
\]
At round $L$, the capped gambler wagers the whole fortune and reaches
zero, while the shadow update gives
\[
  \widehat Y_L=2(\widehat Y_{L-1}-1)\le0.
\]
Thus $\tau=\widehat\tau=L$ on $\{N<\infty\}$, proving the asserted
equality of bankruptcy times.

Since $B_{n+1}>0$, one has $\widehat\tau=\sigma_0$. The interval $(-\infty,2]$ is invariant under both maps in \eqref{eq:Y-recursion}. Starting from any value at most $2$, the outcome word $(+1,-1)$ sends the shadow ratio to a value at most $1$:
\[
  2\left(\frac{y+1}{2}-1\right)=y-1\le1.
\]
After time $\sigma_2$, this word occurs almost surely in one of infinitely
many disjoint two-round blocks, because $p(1-p)>0$. Hence
$\sigma_2<\infty$ implies $\sigma_1<\infty$ almost surely. The reverse
implications among $\sigma_0,\sigma_1,\sigma_2$ follow from the nested
thresholds. Taking complements proves \eqref{eq:cap-equivalence}.
\end{proof}

Put
\[
  X_n:=\widehat Y_n-2.
\]
Then \eqref{eq:Y-recursion} becomes the globally defined auxiliary recursion
\begin{equation}
\label{eq:X-recursion}
  X_{n+1}=
  \begin{cases}
    (X_n-1)/2,&\xi_{n+1}=1,\\
    2X_n,&\xi_{n+1}=-1.
  \end{cases}
\end{equation}
For $n\in\Nzero$, define
\begin{equation}
\label{eq:partial-perpetuity}
  \mathcal S^{(n)}:=\sum_{i=1}^n\one_{\{\xi_i=1\}}2^{R_{i-1}},
  \qquad
  \mathcal S^{(0)}:=0.
\end{equation}
An induction gives
\begin{equation}
\label{eq:X-explicit}
  X_n=2^{-R_n}\bigl(x-2-\mathcal S^{(n)}\bigr).
\end{equation}
Because $\mathcal S^{(n)}$ is nondecreasing,
\begin{equation}
\label{eq:strong-prefix}
  \{\sigma_2>n\}=\{\mathcal S^{(n)}<x-2\},
  \qquad x>2.
\end{equation}
For $0<p<1$, there are infinitely many wins almost surely, so every finite partial sum has a strictly positive tail: $\mathcal S^{(n)}<\mathcal S$. Letting $n\to\infty$ in \eqref{eq:strong-prefix} and using Lemma~\ref{lem:cap-equivalence} proves
\begin{equation}
\label{eq:survival-representation}
  f(x,p)=\PP(\mathcal S\le x-2),
  \qquad x>2,\quad 0<p<1.
\end{equation}
Conditioning on the first outcome also proves the fixed-point identity
\eqref{eq:fixed-point-intro}. Figure~\ref{fig:mechanism} displays
\eqref{eq:X-explicit}--\eqref{eq:survival-representation} along a single
outcome sequence.

\begin{figure}[htbp]
  \centering
  \includegraphics[width=0.92\textwidth]{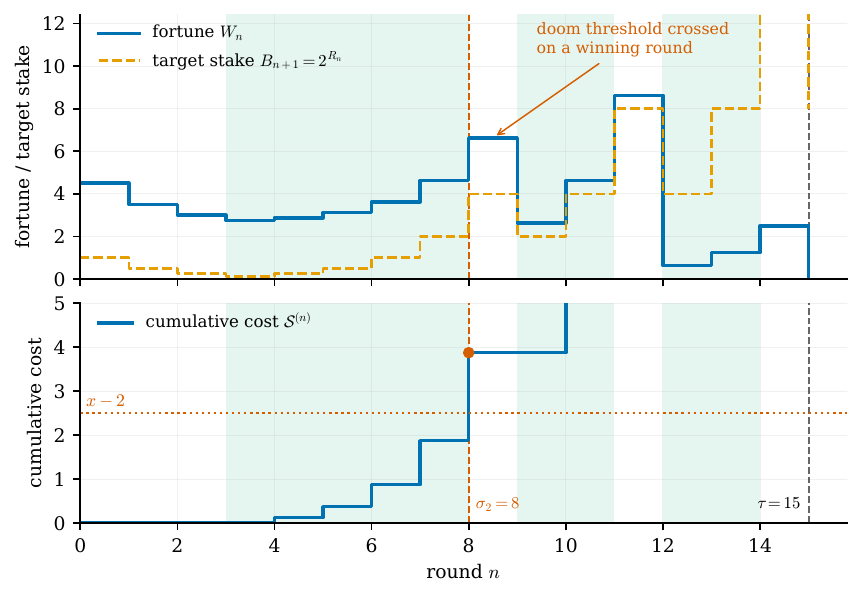}
  \caption{The mechanism of \eqref{eq:X-explicit}--\eqref{eq:survival-representation}
  along one outcome sequence, started from $x=4.5$. Winning rounds are shaded.
  \emph{Top:} the fortune $W_n$ of \eqref{eq:capped-model} and the target bet
  $B_{n+1}=2^{R_n}$. \emph{Bottom:} the partial sum $\mathcal S^{(n)}$ of
  \eqref{eq:partial-perpetuity}, which by \eqref{eq:X-explicit} is the entire
  cost of the game: it increases \emph{only} at wins, never at losses, and by
  \eqref{eq:strong-prefix} the gambler is doomed exactly when it reaches
  $x-2$. Here it does so at $\sigma_2=8$, at the moment when the fortune
  $W_8=6.625$ exceeds the initial fortune $x$; bankruptcy
  itself follows only at $\tau=15$. The picture explains both the threshold
  $x=2$ --- the reserve $2$ is what an infinite losing streak costs --- and
  Corollary~\ref{cor:strict-p}: under the monotone coupling, raising $p$ can only increase
  the cumulative cost in \eqref{eq:perpetuity-intro}.}
  \label{fig:mechanism}
\end{figure}

\begin{proof}[Proof of Theorem~\ref{thm:phase}]
If $x\le2$, then $\sigma_2=0$, and Lemma~\ref{lem:cap-equivalence} gives $f(x,p)=0$.

Suppose $x>2$. If $p>1/2$, the strong law gives $R_n/n\to2p-1>0$ almost surely. Wins occur infinitely often, and the corresponding summands in \eqref{eq:perpetuity-intro} fail to tend to zero; hence $\mathcal S=\infty$ almost surely. If $p=1/2$, the simple symmetric random walk makes the transition $0\to1$ infinitely often almost surely. Each such transition contributes $1$ to $\mathcal S$, so again $\mathcal S=\infty$ almost surely. In either case, \eqref{eq:survival-representation} yields $f(x,p)=0$.

Now let $p<1/2$. The strong law gives $R_n/n\to2p-1<0$, so the series \eqref{eq:perpetuity-intro} converges almost surely. Put $t=x-2>0$, and let $\mathcal S'$ be an independent copy of $\mathcal S$. On the event that the first $m$ rounds are losses,
\[
  \mathcal S=2^{-m}\mathcal S',
  \qquad\text{in distribution}.
\]
Since $\mathcal S'<\infty$ almost surely, choose $m$ with $\PP(\mathcal S'\le2^mt)>0$. Then
\[
  \PP(\mathcal S\le t)
  \ge(1-p)^m\PP(\mathcal S'\le2^mt)>0.
\]
Equation \eqref{eq:survival-representation} now gives $f(x,p)>0$.
Choose $K$ so that $2^K-1>t$. If the first $K$ rounds are wins, then $\mathcal S\ge2^K-1>t$; hence $f(x,p)\le1-p^K<1$.
\end{proof}

\begin{proof}[Proof of Proposition~\ref{prop:approximants}]
For $x>2$, equations \eqref{eq:strong-prefix} and
\eqref{eq:strong-survival-definition} give
\begin{equation}
\label{eq:finite-prefix-representation}
  f_n(x,p)=\PP_p(\mathcal S^{(n)}<x-2).
\end{equation}
The events on the right decrease to $\{\mathcal S\le x-2\}$: the
strict inequality at every finite stage includes the possible event
$\{\mathcal S=x-2\}$. Thus \eqref{eq:survival-representation} and
continuity from above give $f_n(x,p)\downarrow f(x,p)$. The same
conclusion is immediate for $x\le2$.

Taking limits in \eqref{eq:finite-recursion} proves
\eqref{eq:self-similarity}. For $p<1/2$, the zero extension satisfies the
same equation for $x\le2$, while almost sure finiteness of $\mathcal S$
gives $f(x,p)\to1$ as $x\to\infty$.

For fixed $n$, \eqref{eq:finite-prefix-representation} shows that
$f_n(\cdot,p)$ is a nondecreasing, left-continuous step function. Summing
over the $2^n$ outcome words makes $f_n(x,\cdot)$ a polynomial, and the
usual monotone coupling shows that it is nonincreasing in $p$. Finally,
if $v:\R\to[0,1]$ vanishes on $(-\infty,2]$ and
$\mathcal L_pv=v$, then $v\le f_0$ and positivity of $\mathcal L_p$
gives
\[
  v=\mathcal L_p^nv\le\mathcal L_p^nf_0=f_n.
\]
Letting $n\to\infty$ proves $v\le f$.
\end{proof}

\section{Regularity in the initial fortune}
\label{sec:x-regularity}

Throughout this section, $p\in(0,1/2)$ is fixed.

We first record a ladder-block representation. For $i\in\Nzero$, let
\[
  \vartheta_i:=\inf\{n\ge0:R_n=-i\},
\]
so $\vartheta_0=0$ and every $\vartheta_i$ is finite almost surely. Define
\begin{equation}
\label{eq:ladder-blocks}
  A_i:=\sum_{n=\vartheta_i+1}^{\vartheta_{i+1}}
  \one_{\{\xi_n=1\}}2^{R_{n-1}+i}.
\end{equation}

\begin{lemma}
\label{lem:ladder-blocks}
The random variables $(A_i)_{i\in\Nzero}$ are independent, identically distributed, and $\Nzero$-valued. Moreover,
\begin{equation}
\label{eq:S-ladder-series}
  \mathcal S=\sum_{i=0}^{\infty}A_i2^{-i}
  \qquad\text{almost surely},
\end{equation}
and
\begin{equation}
\label{eq:full-support-A}
  \PP(A_0=m)\ge p^m(1-p)^{m+1}>0,
  \qquad m\in\Nzero.
\end{equation}
Also $\mathcal S>0$ almost surely.
\end{lemma}

\begin{proof}
The strong law gives $R_n\to-\infty$. The strong Markov property at the descending ladder times shows that the blocks in \eqref{eq:ladder-blocks} are independent and identically distributed. Before $\vartheta_{i+1}$, the shifted walk is nonnegative, so every summand defining $A_i$ is an integer. Grouping the nonnegative terms in \eqref{eq:perpetuity-intro} by ladder blocks proves \eqref{eq:S-ladder-series}.

For $m\ge0$, the path $(+1,-1)^m,-1$ before the first visit to $-1$ has probability $p^m(1-p)^{m+1}$ and makes $A_0=m$. This proves \eqref{eq:full-support-A}, including the case $m=0$. Finally, a win occurs almost surely, so $\mathcal S>0$.
\end{proof}

Write $q:=1-p$, let $A$ have the common law of the ladder digits, and set
\[
  \phi_p(z):=\EE z^A,
  \qquad |z|\le1.
\]

\begin{lemma}[Ladder-digit recursion]
\label{lem:digit-recursion}
If $A',A''$ are independent copies of $A$, then
\begin{equation}
\label{eq:digit-recursion}
  A\ \stackrel{\mathrm d}{=}\
  \begin{cases}
    0,&\text{with probability }q,\\
    1+2A'+A'',&\text{with probability }p.
  \end{cases}
\end{equation}
Consequently,
\begin{equation}
\label{eq:digit-pgf}
  \phi_p(z)=\frac{q}{1-pz\phi_p(z^2)},
\end{equation}
and
\begin{equation}
\label{eq:digit-parity}
  \PP(A\text{ is even})=\frac1{1+p},
  \qquad
  \PP(A\text{ is odd})=\frac p{1+p}.
\end{equation}
\end{lemma}

\begin{proof}
If the first step of a ladder block is a loss, the digit is zero. If it is a win, the walk first returns from level $1$ to level $0$ and then makes a fresh block from $0$ to $-1$. Translation of the first piece multiplies all its weights by two, which gives \eqref{eq:digit-recursion}. Taking generating functions and moving the factor $\phi_p(z)$ to the left gives \eqref{eq:digit-pgf}. Since $\phi_p(1)=1$, substitution of $z=-1$ gives $\phi_p(-1)=q/(1+p)$; taking the even and odd parts proves \eqref{eq:digit-parity}.
\end{proof}

For $n\ge1$, define the integer prefix and its largest atom by
\begin{equation}
\label{eq:prefix-Z-M}
  Z_n:=\sum_{i=0}^{n-1}2^{n-1-i}A_i,
  \qquad
  M_n:=\sup_{z\in\mathbb Z}\PP(Z_n=z).
\end{equation}
We use the conventions $Z_0=0$ and $M_0=1$.

The following elementary two-step calculation is the only
anti-concentration estimate that we shall need.

\begin{lemma}[Two-step concentration]
\label{lem:two-step-concentration}
Let $A_0,A_1$ be independent copies of $A$ and put $B=2A_0+A_1$.
Then
\begin{equation}
\label{eq:rho-two}
  \max_{b\bmod4}\PP(B\equiv b\pmod4)
  =\rho(p)
  =\frac{1+q^2/(1+c^2)}{2(1+p)},
  \qquad c=\frac{pq}{1+p}.
\end{equation}
Moreover,
\begin{equation}
\label{eq:M-two-step}
  M_n\le \rho(p)^{n/2}=2^{-n\alpha(p)},
  \qquad n\ge0.
\end{equation}
\end{lemma}

\begin{proof}
Set
\[
  d:=\phi_p(-1)=\frac{q}{1+p}.
\]
Using \eqref{eq:digit-pgf} at $z=i$ gives
\[
  \phi_p(i)=\frac{q}{1-ic}=u+iv,
  \qquad
  u=\frac{q}{1+c^2},
  \quad v=cu.
\]
Fourier inversion on $\mathbb Z/4\mathbb Z$ gives, in order for the
residues $0,1,2,3$, the four probabilities
\[
  \frac{1+d+2du}{4},\qquad
  \frac{1-d+2dv}{4},\qquad
  \frac{1+d-2du}{4},\qquad
  \frac{1-d-2dv}{4}.
\]
The differences between the first probability and the other three are,
respectively,
\[
  \frac{d(1+u-v)}2,\qquad du,\qquad \frac{d(1+u+v)}2.
\]
They are positive because $d,u,v>0$ and $v<u$, so the first probability
is the unique maximum. It simplifies to \eqref{eq:rho-two}.

If $B'$ is independent of $Z_n$ and distributed as $B$, then
\[
  Z_{n+2}=4Z_n+B'.
\]
For any integer $z$, the convolution formula therefore gives
\[
  \PP(Z_{n+2}=z)
  \le M_n\PP(B'\equiv z\pmod4)
  \le \rho(p)M_n.
\]
Thus $M_{n+2}\le\rho(p)M_n$. Also $M_1=q$: indeed,
$\PP(A=0)=q$, whereas every positive atom of $A$ is at most
$\PP(A>0)=p<q$. Finally, the event $A_0=A_1=0$ shows that
$\rho(p)\ge q^2$. Consequently,
\[
  M_{2m}\le\rho(p)^m,
  \qquad
  M_{2m+1}\le q\rho(p)^m\le\rho(p)^{m+1/2},
\]
which is \eqref{eq:M-two-step}.
\end{proof}

\begin{proof}[Proof of Theorem~\ref{thm:holder}]
We first establish the exponent at $2$ directly from the functional
equation. Given $y\in(0,1)$, choose $n\ge1$ with
$2^{-n}\le y<2^{-n+1}$. Iterating
$F_p(t)=qF_p(2t)$ gives
\begin{equation}
\label{eq:boundary-scaling}
  f(2+y,p)=q^nF_p(2^ny),
  \qquad 1\le2^ny<2.
\end{equation}
Since $F_p$ is nondecreasing and $F_p(1)=f(3,p)>0$, writing
$a(p):=\log_{1/2}q=-\log_2q$ yields
\[
  q\,f(3,p)\,y^{a(p)}
  \le f(2+y,p)\le y^{a(p)}.
\]
This proves \eqref{eq:boundary-exponent}.

We next prove strict positivity in \eqref{eq:holder-bound}. Put
$y=x-2\ge0$. Choose $K$ so that $2^{1-K}<r$ and set
\[
  y_K:=2^{-K}\lceil2^Ky\rceil.
\]
Then $y\le y_K<y_K+2^{-K}<y+r$. Write the terminating dyadic
expansion
\[
  y_K=\sum_{j=0}^Ka_j2^{-j},
  \qquad
  a_0\in\Nzero,\quad a_1,\ldots,a_K\in\{0,1\}.
\]
By Lemma~\ref{lem:ladder-blocks} and independence,
\begin{align*}
  \PP(y<\mathcal S\le y+r)
  &\ge
  \PP\left(
  A_j=a_j,\ 0\le j\le K;
  \ 0<\sum_{j>K}A_j2^{-j}\le2^{-K}
  \right)\\
  &=\left(\prod_{j=0}^K\PP(A_j=a_j)\right)
    \PP(0<\mathcal S'\le2)\\
  &=\left(\prod_{j=0}^K\PP(A_j=a_j)\right)f(4,p)>0,
\end{align*}
where $\mathcal S'$ is an independent copy of $\mathcal S$; explicitly,
\[
  \sum_{j>K}A_j2^{-j}\ \stackrel{\mathrm d}{=}\
  2^{-(K+1)}\mathcal S'.
\]

For the upper bound, the ladder decomposition gives the exact identity
\begin{equation}
\label{eq:scaled-tail}
  2^{n-1}\mathcal S=Z_n+\frac{\mathcal S'}2,
\end{equation}
with $\mathcal S'$ independent of $Z_n$. If $0<r<1$, choose $n\ge1$
so that $2^{-n}\le r<2^{-n+1}$. Conditional on $\mathcal S'$, the
event
\[
  x-2<\mathcal S\le x-2+r
\]
places the integer $Z_n$ in an interval of length
$2^{n-1}r<1$. Lemma~\ref{lem:two-step-concentration} therefore gives
\[
  f(x+r,p)-f(x,p)
  \le M_n
  \le2^{-n\alpha(p)}
  \le r^{\alpha(p)}.
\]
For $r\ge1$, the same inequality is immediate because its left-hand
side is at most one.

It remains to verify \eqref{eq:holder-exponent-comparison}. Direct
simplification gives
\begin{equation}
\label{eq:rho-strict-improvement}
  \frac1{(1+p)^2}-\rho(p)
  =\frac{p^2q\bigl((1+p)^2+q^2\bigr)}
  {2(1+p)^2\bigl((1+p)^2+p^2q^2\bigr)}>0.
\end{equation}
Hence $\alpha(p)>\log_2(1+p)$. Substitution of $p=q=1/2$ into
\eqref{eq:rho-two} gives $\rho(p)\to46/111$, and the asserted limit
follows.
\end{proof}

\begin{remark}[The boundary profile]
\label{rem:profile}
Iterating $F_p(y)=qF_p(2y)$ rather than taking logarithms turns
\eqref{eq:boundary-exponent} into an identity: for $0<y\le2$,
\[
  f(2+y,p)=y^{a(p)}\Pi_p(y),
  \qquad \Pi_p(2y)=\Pi_p(y),
\]
where $\Pi_p(y):=f(2+y,p)y^{-a(p)}$ on $(1,2]$, extended to $(0,\infty)$ by
$\Pi_p(y)=\Pi_p(2^ky)$, is continuous and positive; the identity
$q\,2^{a(p)}=1$ is what makes this consistent. The function $\Pi_p$ is not
constant: otherwise $f(2+y,p)=Cy^{a(p)}$ on $(1,2)$ would have a strictly
positive derivative there, contrary to Theorem~\ref{thm:regularity}(1).
Thus \eqref{eq:boundary-exponent} cannot be upgraded to
$f(2+y,p)\sim Cy^{a(p)}$ as $y\downarrow0$: such an asymptotic would force
the periodic function $\Pi_p$ to be constant. See Figure~\ref{fig:profile}.
\end{remark}

\begin{figure}[htbp]
  \centering
  \includegraphics[width=0.78\textwidth]{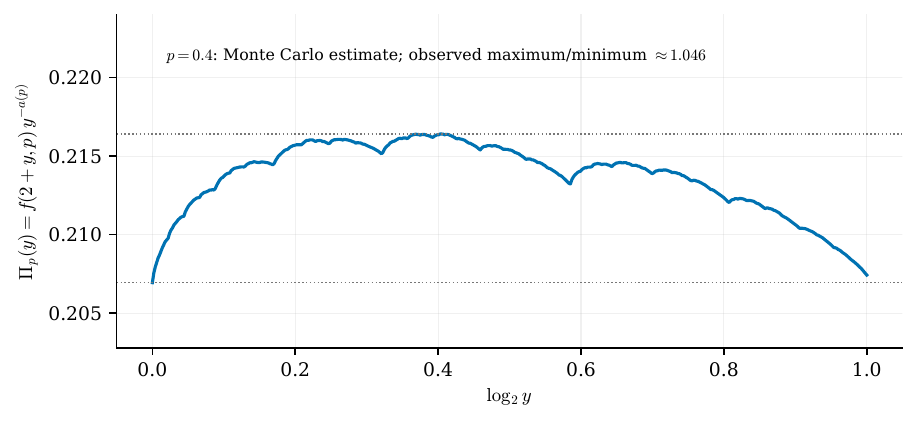}
  \caption{Monte Carlo approximation to one period of the boundary profile
  $\Pi_p(y)=f(2+y,p)y^{-a(p)}$ at $p=0.4$, based on $2\times10^6$ samples
  of a truncated series. Here $a(p)=-\log_2 0.6=0.7370\ldots$ and the
  horizontal axis is $\log_2y$. The plotted interval $1\le y\le2$
  determines the full profile by multiplicative periodicity.
  Nonconstancy is proved in Remark~\ref{rem:profile}.}
  \label{fig:profile}
\end{figure}

\begin{proof}[Proof of Theorem~\ref{thm:regularity}(1)]
Theorem~\ref{thm:holder} proves continuity and strict increase. Let $\nu$
be the law of $\mathcal S$. Its distribution function is $F_p$, which is
continuous, so $\nu$ is atomless.

Let $\overline\nu$ be the law of the fractional part $\{\mathcal S\}$
on $[0,1)$ and let $T(t)=2t\pmod1$. From
\eqref{eq:S-ladder-series}, deleting the first ladder digit gives
\[
  T(\{\mathcal S\})
  =\left\{\sum_{j=1}^{\infty}A_j2^{1-j}\right\}.
\]
Thus the map from the i.i.d. ladder sequence to $\{\mathcal S\}$
intertwines the Bernoulli shift with $T$. It follows that
$\overline\nu$ is $T$-invariant and ergodic. This is the classical
stationary ergodic digit mechanism studied by Dym \cite{Dym1968}; see
Cornean et al.~\cite{CorneanEtAl2022} for a modern structural treatment.

The measure $\overline\nu$ is not Lebesgue measure. Indeed, if
$A_0=\cdots=A_{k-1}=0$ and an independent copy $\mathcal S'$ is at
most one, then $\mathcal S=2^{-k}\mathcal S'\le2^{-k}$. Hence
\begin{equation}
\label{eq:small-mod-one-mass}
  \overline\nu([0,2^{-k}])
  \ge\PP(\mathcal S\le2^{-k})
  \ge q^kf(3,p)>2^{-k}
\end{equation}
for all sufficiently large $k$. Lebesgue measure on $[0,1)$ is also
ergodic for $T$, and two distinct ergodic invariant probability measures
are mutually singular. Therefore there is a Lebesgue-null set
$E\subset[0,1)$ with $\overline\nu(E)=1$. The set
$\bigcup_{m\ge0}(m+E)$ is Lebesgue-null and has full $\nu$-measure, so
$\nu\perp\Leb$.

The standard differentiation theorem for distribution functions now gives
$F_p'=0$ almost everywhere; see, for example,
\cite[Proposition~3.30]{Folland1999}. Hence $f(\cdot,p)=F_p(\cdot-2)$ is
singular continuous on $[2,\infty)$ and extends constantly as zero on
$[1,2]$.
\end{proof}

\section{Monotonicity and analyticity in the winning probability}
\label{sec:p-regularity}

\begin{proof}[Proof of Corollary~\ref{cor:strict-p}]
Fix $x>2$ and $0<p_1<p_2<1/2$, and put
\[
  r:=\frac{p_2-p_1}{1-p_1}.
\]
Let $(U_n)_{n\ge1}$ and $(V_n)_{n\ge1}$ be independent sequences of i.i.d. uniform random variables on $(0,1)$, and define
\begin{align*}
  \xi_n^{(1)}&:=2\one_{\{U_n<p_1\}}-1,\\
  \xi_n^{(2)}&:=2\one_{\{U_n<p_1\ \text{or}\ (U_n\ge p_1,\ V_n<r)\}}-1.
\end{align*}
Then $(\xi_n^{(i)})$ has parameter $p_i$, and $\xi_n^{(1)}\le\xi_n^{(2)}$ for every $n$.

For a deterministic sequence $a=(a_n)_{n\ge1}\in\{-1,1\}^{\N}$, set
\[
  \mathcal S(a):=\sum_{n\ge1}\one_{\{a_n=1\}}
  2^{\sum_{j=1}^{n-1}a_j}.
\]
This functional is coordinatewise nondecreasing. Indeed, if $a_n\le b_n$ for
every $n$, then every winning summand of $\mathcal S(a)$ also occurs in
$\mathcal S(b)$, and its prefix exponent is no smaller for $b$ than for $a$.
The series have nonnegative terms, so termwise comparison gives
\[
  \mathcal S^{(1)}:=\mathcal S(\xi^{(1)})
  \le\mathcal S^{(2)}:=\mathcal S(\xi^{(2)})
  \qquad\text{almost surely}.
\]

Choose $K$ so that $2^K-1>x-2$, and let $G_K:=\bigcap_{n=1}^K\{V_n<r\}$. On $G_K$, the first $K$ variables $\xi_n^{(2)}$ are all $1$, so $\mathcal S^{(2)}\ge2^K-1>x-2$. Moreover, $G_K$ is independent of $\mathcal S^{(1)}$. By \eqref{eq:survival-representation} and Theorem~\ref{thm:phase},
\begin{align*}
  f(x,p_1)-f(x,p_2)
  &=\PP(\mathcal S^{(1)}\le x-2,\ \mathcal S^{(2)}>x-2)\\
  &\ge\PP(G_K)\PP(\mathcal S^{(1)}\le x-2)\\
  &=r^Kf(x,p_1)>0.\qedhere
\end{align*}
\end{proof}

Our analyticity proof complexifies the Bernoulli parameter and controls
a series of finite-prefix increments. This is the same general scheme
used for Lyapunov exponents by Peres~\cite{Peres1991,Peres1992} and for
rapidly determined tree properties by Peres and Swan~\cite{PeresSwan2020};
those theorems do not directly concern the survival probability here.

\begin{lemma}[Uniform finite-prefix approximation]
\label{lem:uniform-approximation}
For every compact interval $K\subset(0,1/2)$ there are constants
$A<\infty$ and $\theta\in(0,1)$ such that
\begin{equation}
\label{eq:uniform-prefix-approximation}
  0\le f_n(x,p)-f(x,p)\le A\theta^n,
  \qquad p\in K,\quad x\ge2,\quad n\ge0.
\end{equation}
\end{lemma}

\begin{proof}
Write $K=[a,b]$. Choose $s\in(0,1]$ so that
\[
  s<\log_2\frac{1-b}{b},
  \qquad
  r:=b2^s+(1-b)2^{-s}<1,
\]
and put $\gamma:=\log_2(1+a)>0$. For $p\in K$, set
$r_p:=p2^s+(1-p)2^{-s}\le r$. Since $s\le1$, subadditivity and
independence give
\begin{equation}
\label{eq:fractional-moment}
\begin{aligned}
  \EE_p\mathcal S^s
  &\le \sum_{j\ge0}
    \EE_p\left(\one_{\{\xi_{j+1}=1\}}2^{R_j}\right)^s
   =\sum_{j\ge0}p r_p^j
   \le C:=\frac{b}{1-r}.
\end{aligned}
\end{equation}
The tail after $n$ outcomes has the representation
\[
  \mathcal S-\mathcal S^{(n)}=2^{R_n}\mathcal S',
\]
where $\mathcal S'$ is independent of the first $n$ outcomes and has the
law of $\mathcal S$. Therefore
\begin{equation}
\label{eq:fractional-tail}
  \EE_p(\mathcal S-\mathcal S^{(n)})^s
  \le Cr^n.
\end{equation}

Let $t=x-2>0$ and $0<\varepsilon\le1$. By
\eqref{eq:finite-prefix-representation}, Markov's inequality, and
Theorem~\ref{thm:holder},
\begin{align*}
  0\le f_n(x,p)-f(x,p)
  &=\PP_p(\mathcal S^{(n)}<t<\mathcal S)\\
  &\le \PP_p(t<\mathcal S\le t+\varepsilon)
     +\PP_p(\mathcal S-\mathcal S^{(n)}>\varepsilon)\\
  &\le \varepsilon^\gamma+Cr^n\varepsilon^{-s}.
\end{align*}
Here we used $\alpha(p)>\log_2(1+p)\ge\gamma$. Taking
$\varepsilon=r^{n/(s+\gamma)}$ proves
\eqref{eq:uniform-prefix-approximation} with
\[
  A=1+C,
  \qquad
  \theta=r^{\gamma/(s+\gamma)}<1.
\]
For $x=2$, both sides of the difference are zero.
\end{proof}

\begin{proof}[Proof of Theorem~\ref{thm:regularity}(2)]
Fix $x>2$ and a compact interval $K=[a,b]\subset(0,1/2)$. For
$n\ge0$, set
\[
  u_n(p):=f_n(x,p),
  \qquad
  g_n(p):=u_n(p)-u_{n+1}(p)
  =\PP_p(X_n>0,X_{n+1}\le0).
\]
Since $u_0=1$,
\begin{equation}
\label{eq:f-telescope}
  u_n(p)=1-\sum_{j=0}^{n-1}g_j(p).
\end{equation}
Lemma~\ref{lem:uniform-approximation} gives constants $A<\infty$ and
$\theta\in(0,1)$ such that
\begin{equation}
\label{eq:g-decay}
  0\le g_n(p)\le u_n(p)-f(x,p)\le A\theta^n,
  \qquad p\in K.
\end{equation}

The event defining $g_n$ depends on the first $n+1$ outcomes. If
$C_j(n)$ is the number of its outcome words having exactly $j$ wins, then
\begin{equation}
\label{eq:g-polynomial}
  g_n(z)=\sum_{j=0}^{n+1}C_j(n)z^j(1-z)^{n+1-j}
\end{equation}
defines an entire polynomial that agrees with $g_n$ on the real interval.
Choose $\eta>0$ so that $\theta(1+\eta)<1$, and set
\[
  U:=\{z\in\mathbb C:a<\Re z<b,
       \ |\Im z|<\eta\min(a,1-b)\}.
\]
If $z\in U$ and $p=\Re z$, then
\[
  \frac{|z|}{p}\le1+\eta,
  \qquad
  \frac{|1-z|}{1-p}\le1+\eta.
\]
The nonnegative coefficients in \eqref{eq:g-polynomial} and
\eqref{eq:g-decay} now give
\[
  |g_n(z)|
  \le(1+\eta)^{n+1}g_n(p)
  \le A(1+\eta)[\theta(1+\eta)]^n.
\]
Hence $\sum_{n\ge0}g_n(z)$ converges locally uniformly on $U$, and
\[
  u(z):=1-\sum_{n\ge0}g_n(z)
\]
is holomorphic there. For real $p\in(a,b)$,
\eqref{eq:f-telescope} and Proposition~\ref{prop:approximants} give
$u(p)=f(x,p)$. Thus $f(x,\cdot)$ is real analytic on $(a,b)$. Since
every point of $(0,1/2)$ lies in such an interval, the proof is complete.
\end{proof}

\section{Refined analysis of H\"older exponents}
\label{sec:refined}

Theorem~\ref{thm:holder} extracts one explicit exponent from a two-step
concentration estimate. This section places it inside an infinite hierarchy of
exponents proved by the same argument, gives an exact formula for the
concentration function $M_n$ of \eqref{eq:prefix-Z-M}, and uses that formula to
show that the natural guess for the optimal exponent is false. Throughout
$p\in(0,1/2)$, $q=1-p$, $a(p):=\log_{1/2}q=-\log_2(1-p)$ as in the proof of
Theorem~\ref{thm:holder}, and $\beta^*(p)$ denotes the supremum of those
$\beta>0$ for which $f(\cdot,p)$ is $\beta$-H\"older on $[2,\infty)$. Since $f(2,p)=0$,
the boundary asymptotics \eqref{eq:boundary-exponent} force
$\beta^*(p)\le a(p)$, since $f(2+y,p)\le Cy^\beta$ with $\beta>a(p)$ would
contradict $\log f(2+y,p)/\log y\to a(p)$, while Theorem~\ref{thm:holder} gives
$\beta^*(p)\ge\alpha(p)$. Thus $a(p)$ is a ceiling for the optimal exponent,
and the question is whether it is attained.

Lemma~\ref{lem:two-step-concentration} rests on the identity
$Z_{n+2}=4Z_n+(2A_n+A_{n+1})$. Nothing distinguishes the number two.

\begin{lemma}[$m$-step concentration]
\label{lem:m-step}
For all integers $m\ge1$ and $n\ge0$ one has $Z_{n+m}=2^mZ_n+Z_m'$, where $Z_m'$ is independent of
$Z_n$ and distributed as $Z_m$. Consequently, with
\[
  \rho_m(p):=\max_{b\bmod 2^m}\PP\bigl(Z_m\equiv b\ (\mathrm{mod}\ 2^m)\bigr),
\]
one has $M_{n+m}\le\rho_m(p)M_n$, hence $M_n\le\rho_m(p)^{\lfloor n/m\rfloor}$
and $\rho_{m+m'}\le\rho_m\rho_{m'}$ for $m,m'\ge1$, and for all $x\ge2$ and $r>0$,
\begin{equation}
\label{eq:alpha-m}
  f(x+r,p)-f(x,p)\ \le\ C_m(p)\,r^{\alpha_m(p)},
  \qquad
  \alpha_m(p):=-\frac{\log_2\rho_m(p)}{m},
  \quad C_m(p):=\rho_m(p)^{-\frac{m-1}{m}} .
\end{equation}
Each $\rho_m$ is explicit: Fourier inversion on $\mathbb Z/2^m\mathbb Z$ gives
\begin{equation}
\label{eq:rho-m-fourier}
  \PP\bigl(Z_m\equiv b\bigr)
  =2^{-m}\sum_{k=0}^{2^m-1}\omega^{-kb}\prod_{j=0}^{m-1}
   \phi_p\bigl(\omega^{k2^j}\bigr),
  \qquad\omega=e^{2\pi i/2^m},
\end{equation}
and the values of $\phi_p$ at $2^m$-th roots of unity are obtained from
$\phi_p(1)=1$ by $m$ applications of \eqref{eq:digit-pgf}, each a single
division. Finally $\alpha_m(p)\to\sup_{m\ge1}\alpha_m(p)$ as $m\to\infty$.
\end{lemma}

\begin{proof}
Splitting the sum defining $Z_{n+m}$ at the index $n$ gives
$Z_{n+m}=2^mZ_n+\sum_{j<m}2^{m-1-j}A_{n+j}$, and the second sum is an
independent copy of $Z_m$. Hence, for any $z$,
\[
  \PP(Z_{n+m}=z)=\sum_{z'}\PP(Z_m'=z')\,\PP\bigl(Z_n=(z-z')/2^m\bigr)
  \le M_n\!\!\sum_{z'\equiv z\,(2^m)}\!\!\PP(Z_m'=z')\le\rho_m M_n .
\]
Iterating and using $M_j\le1$ gives $M_n\le\rho_m^{\lfloor n/m\rfloor}
\le C_m\rho_m^{n/m}=C_m2^{-n\alpha_m}$, and \eqref{eq:alpha-m} follows exactly
as in the proof of Theorem~\ref{thm:holder}: for $0<r<1$ choose $n$ with
$2^{-n}\le r<2^{1-n}$, so that \eqref{eq:scaled-tail} confines $Z_n$,
conditionally on $\mathcal S'$, to an interval of length $2^{n-1}r<1$ and
$f(x+r,p)-f(x,p)\le M_n\le C_mr^{\alpha_m}$; for $r\ge1$ the bound is trivial
since $f\le1\le C_m$. The same splitting applied to
$Z_{m+m'}=2^{m'}Z_m+Z_{m'}'$ gives $\rho_{m+m'}\le\rho_m\rho_{m'}$, whence the
convergence of $\alpha_m$ by Fekete's lemma. Formula \eqref{eq:rho-m-fourier} is Fourier inversion together with
$\EE\omega^{kZ_m}=\prod_{j<m}\phi_p(\omega^{k2^j})$.
\end{proof}

The case $m=2$ is Theorem~\ref{thm:holder}, where the sharper bookkeeping of
Lemma~\ref{lem:two-step-concentration} removes the constant. The exponents converge to their supremum and satisfy
$\alpha_{km}(p)\ge\alpha_m(p)$ for all $k,m\ge1$. In particular, they
are nondecreasing along $m=2,4,8,16$ (Figure~\ref{fig:exponents}): at $p=0.4$, for instance,
$\alpha_2=0.5264$, $\alpha_4=0.5813$ and $\alpha_{16}=0.6829$, against the
ceiling $a(p)=0.7370$; and the limit $\tfrac12\log_2(111/46)=0.6354\ldots$ in
\eqref{eq:holder-exponent-comparison} is an artefact of the choice $m=2$, since
$\alpha_4(p)\to0.7031\ldots$ and $\alpha_{16}(p)\to0.8439\ldots$ as
$p\uparrow1/2$.

\begin{figure}[htbp]
  \centering
  \includegraphics[width=0.78\textwidth]{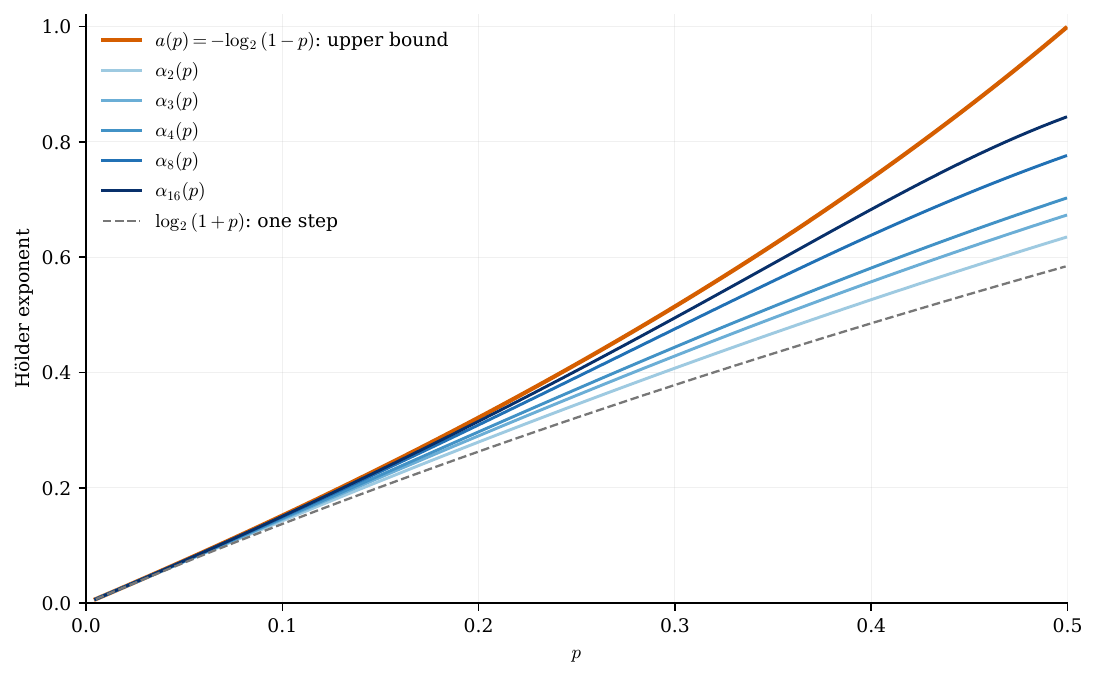}
  \caption{The hierarchy of Lemma~\ref{lem:m-step}: each
  $\alpha_m(p)=-\log_2\rho_m(p)/m$, computed from \eqref{eq:rho-m-fourier}, is
  a H\"older exponent valid for all $x\ge2$ and $r>0$, and $\alpha_2$ is
  Theorem~\ref{thm:holder}. The dashed curve is the one-step benchmark
  $\log_2(1+p)$; the upper solid curve is the ceiling $a(p)$, the pointwise
  exponent at $x=2$. The optimal exponent $\beta^*(p)$ lies in the gap, which
  by Proposition~\ref{prop:certificate} does not close from above once $p$ is
  near $1/2$.}
  \label{fig:exponents}
\end{figure}

\medskip\noindent\emph{An exact formula for the concentration function.}
All of the above is governed by a single power series. Recall
$\phi_p(z)=\EE z^A$ from \eqref{eq:digit-pgf} and put $V:=\phi_p/q$, so that
\eqref{eq:digit-pgf} becomes
\begin{equation}
\label{eq:V-eq}
  V(x)=\frac1{1-pq\,x\,V(x^2)} ;
\end{equation}
in particular $V$, and everything derived from it below, depends on $p$ only
through the product $pq$. Let $\mathcal F\subset\Nzero$ be the set of integers
whose binary expansion has no two adjacent digits equal to $1$, and let $b(z)$
be the number of ones in the binary expansion of $z$.

\begin{proposition}
\label{prop:fibbinary}
Define $Q(x):=\sum_{z\in\mathcal F}(-pq)^{b(z)}x^z$. Then $Q$ is the unique
power series with $Q(0)=1$ satisfying
\begin{equation}
\label{eq:Q-eq}
  Q(x)=Q(x^2)-pq\,x\,Q(x^4),
\end{equation}
its reciprocal has nonnegative coefficients,
\[
  \frac1{Q(x)}=\prod_{j\ge0}V\bigl(x^{2^j}\bigr)=\sum_{z\ge0}N(z)x^z,
  \qquad N(0)=1,
\]
and for every $n\ge0$,
\begin{equation}
\label{eq:coefficients}
  \PP(Z_n=z)=q^n\,N(z)\qquad\text{for }0\le z<2^n .
\end{equation}
Consequently $M_n\ge q^n$ for every $n$, the ratio $M_n/q^n$ is nondecreasing
in $n$, and
\begin{equation}
\label{eq:sandwich}
  \max_{z<2^n}N(z)\ \le\ \frac{M_n}{q^n}\ \uparrow\ \sup_{z\ge0}N(z) .
\end{equation}
In particular $M_n\asymp q^n$, which by the ceiling $a(p)$ is the best possible
rate, holds if and only if $\sup_zN(z)<\infty$.
\end{proposition}

\begin{proof}
Since $Z_n=\sum_{j<n}2^jA_{n-1-j}$ with i.i.d.\ digits,
$\EE x^{Z_n}=\prod_{j<n}\phi_p(x^{2^j})=q^nP_n(x)$ with
$P_n:=\prod_{j<n}V(x^{2^j})$. As $V(0)=1$, the coefficients of $P_n$ increase
with $n$ to those of $P_\infty:=\prod_{j\ge0}V(x^{2^j})$, while
$P_n(x)=P_\infty(x)Q(x^{2^n})$ with $Q:=1/P_\infty$; comparing coefficients of
$x^z$ for $z<2^n$ gives \eqref{eq:coefficients} and hence
\eqref{eq:sandwich}. Next, $P_\infty(x)=V(x)P_\infty(x^2)$, so by
\eqref{eq:V-eq}
\[
  Q(x)=\frac{1-pq\,x\,V(x^2)}{P_\infty(x^2)}
      =Q(x^2)-pq\,x\,\frac{V(x^2)}{P_\infty(x^2)}
      =Q(x^2)-pq\,x\,Q(x^4),
\]
using $V(x^2)=P_\infty(x^2)/P_\infty(x^4)$. Comparing coefficients in
\eqref{eq:Q-eq}, the coefficients $e_z$ of $Q$ satisfy $e_0=1$ and, for
$z\ge1$, $e_{2j}=e_j$, $e_{4j+1}=-pq\,e_j$ and $e_{4j+3}=0$. Read from the
least significant end, this deletes a trailing $0$ at no cost, deletes a
trailing block $01$ at cost $-pq$, and annihilates the coefficient at a
trailing block $11$; hence $e_z=0$ unless $z\in\mathcal F$, in which case
$e_z=(-pq)^{b(z)}$.
\end{proof}

Unwinding the product, $N(z)=\sum\prod_jv_{a_j}$ over all representations
$z=\sum_ja_j2^j$ with $a_j\in\Nzero$ and $v_a=\PP(A=a)/q$: a weighted
\emph{binary partition function}. The unweighted count is Mahler's, of size
$\exp(\Theta(\log^2z))$ \cite{Mahler1940,deBruijn1948,Knuth1966}, so whether
these weights keep $N(z)$ bounded is a further question. The certificate
below shows that they do not when $p$ is sufficiently close to $1/2$.

\medskip\noindent\emph{The boundary exponent is not the global one.}
Proposition~\ref{prop:fibbinary} makes $z=0$ conspicuous: $N(0)=1$, and
$\PP(Z_n=0)=q^n$ is realised by the all-loss path. It is tempting to conjecture
that $0$ is the mode of $Z_n$, that is, $N\le1$ everywhere, which by
\eqref{eq:sandwich} would give $M_n=q^n$ and hence
$f(x+r,p)-f(x,p)\le r^{\,a(p)}$ with constant one, an optimal bound by
the ceiling $a(p)$. The following converse of Lemma~\ref{lem:m-step} shows that
a single coefficient exceeding one forces a strict gap between the global
and boundary exponents.

\begin{lemma}[Block repetition]
\label{lem:block}
Let $s\ge1$ and $0\le m_0<2^s$ be integers. Then $M_{st}\ge q^{st}N(m_0)^t$ for every $t\ge1$, and
$\beta^*(p)\le a(p)-(\log_2N(m_0))/s$.
\end{lemma}

\begin{proof}
Split the $st$ digits of $Z_{st}$ into $t$ consecutive blocks of length $s$, so
that $Z_{st}=\sum_{i<t}2^{si}Z_s^{(i)}$ with $Z_s^{(0)},\ldots,Z_s^{(t-1)}$
independent copies of $Z_s$. By \eqref{eq:coefficients} each block takes the
value $m_0$ with probability $q^sN(m_0)$, and on that event $Z_{st}$ equals the
fixed value $m_0\sum_{i<t}2^{si}$; this gives the first bound. For the second,
\eqref{eq:scaled-tail} with $\mathcal S'$ independent of $Z_n$ gives, for every
$z\in\Nzero$,
\[
  f\bigl(2+z2^{1-n}+2^{-n},p\bigr)-f\bigl(2+z2^{1-n},p\bigr)
  \ \ge\ \PP(Z_n=z)\,\PP(\mathcal S'\le1)=\PP(Z_n=z)f(3,p),
\]
since $Z_n=z$ and $\mathcal S'\le1$ force
$\mathcal S\in[z2^{1-n},z2^{1-n}+2^{-n}]$. Hence a bound
$f(x+r,p)-f(x,p)\le Cr^\beta$ valid for all $x\ge2$ and $r>0$ forces
$C2^{-n\beta}\ge f(3,p)M_n$ for all $n$; taking $n=st$ and letting
$t\to\infty$ gives $\beta\le a(p)-(\log_2N(m_0))/s$.
\end{proof}

\Needspace{7\baselineskip}
\begin{corollary}
\label{cor:dichotomy}
Exactly one of the following holds.
\begin{enumerate}
\item $N(z)\le1$ for every $z$. Then $M_n=q^n$ for every $n$, the bound
$f(x+r,p)-f(x,p)\le r^{\,a(p)}$ holds for all $x\ge2$ and $r>0$, and
$\beta^*(p)=a(p)$.
\item $N(m_0)>1$ for some $m_0$. Then $\beta^*(p)<a(p)$ strictly, and moreover
$\sup_zN(z)=\infty$, since $N\bigl(m_0\sum_{i<t}2^{si}\bigr)\ge N(m_0)^t$ for
every $t$.
\end{enumerate}
\end{corollary}

A finite computation can certify the second alternative. Expanding
$V(x)=(1-pq\,xV(x^2))^{-1}$ successively in
$\prod_{j\ge0}V(x^{2^j})$ gives
\begin{equation}
\label{eq:profile-sum}
  N(m)=\sum_{\sum_j2^j\ell_j=m}(pq)^{\sum_j\ell_j}
       \prod_{j\ge0}\binom{\ell_j+\ell_{j+1}}{\ell_j}.
\end{equation}
Indeed, after choosing $\ell_j$, the next factor is
$V(x^{2^{j+1}})^{\ell_j+1}$; its geometric-series expansion supplies
$\binom{\ell_j+\ell_{j+1}}{\ell_{j+1}}$.
The sum is over nonnegative integer profiles $(\ell_j)$. For fixed $m$,
it is finite and has nonnegative coefficients in $pq$.

We also use the coefficient identities at $p=1/2$. At this parameter
successive descending ladder times are still finite almost surely, so
$A_j$, $Z_n$, and $M_n$ are defined, and the algebraic identities above
and the first assertion of Lemma~\ref{lem:block} remain valid.
The survival function itself is identically zero at $p=1/2$.

The accompanying source archive contains the matrix generator, an independent
validator, and the integer iteration for the following certificate.

\begin{proposition}
\label{prop:certificate}
At $p=1/2$,
\[
  N\bigl(2^{1300}\bigr)\ge Q_0:=
  \frac{282252067917510}{281474976710656}>1.
\]
Consequently,
\[
  \varliminf_{n\to\infty}-\frac1n\log_2M_n
  \le 1-\frac{\log_2Q_0}{1301}
  =0.999996942752\ldots<1.
\]
Moreover, $N(2^{1300})>1$ for every $0.49998\le p<1/2$, and hence
$\beta^*(p)<a(p)$ throughout this interval.
\end{proposition}

\begin{proof}
For $m=2^r$, write
$\ell_j=b_j+2c_{j+1}-c_j$, where $c_0=c_{r+1}=0$, $b_r=1$,
and $b_j=0$ for $j<r$. At $p=1/2$, retain only carry pairs in
\[
 E:=\{(u,v)\in\{0,\ldots,1024\}^2:
 u\le2v,\ |u-v|\le\lceil5\sqrt{v+1}\rceil\}.
\]
For $(u,v)\in E$, set $a=2v-u$ and define
\[
 A_{u,v}=\mathbf1_{\{u=0\}}2^{-2v},\qquad
 T_{(v,w),(u,v)}=\binom{a+2w-v}{a}2^{-a-(2w-v)},
\]
where both index pairs lie in $E$ and all other entries are zero. Put
\[
 E_{u,v}^{\rm term}=
 \begin{cases}
 \binom{a+1-v}{a}2^{-a-2(1-v)},&v\le1,\\
 0,&v>1.
 \end{cases}
\]
The retained sum in \eqref{eq:profile-sum} is
$(E^{\rm term})^{\mathsf T}T^{r-1}A$.
The supplied code rounds every entry downward to a multiple of $2^{-48}$
and uses exact integer arithmetic, rounding downward after each matrix-vector
product and the terminal scalar product. Positivity makes the resulting
rational number a lower bound. For $r=1300$, the calculation gives $Q_0$;
the validator checks all $208654$ states and $42821618$ nonzero matrix entries
by an independent binomial recurrence. The first inequality for $M_n$ now
follows from Lemma~\ref{lem:block} with $s=1301$.

For the explicit interval, let $R_t$ denote the retained sum with $t=pq$.
Every retained profile has degree
\[
 \sum_{j=0}^r\ell_j=1+\sum_{j=1}^r c_j\le
 D:=1+1300\cdot1024=1331201.
\]
Thus $R_t\ge(4t)^DR_{1/4}$ for $0<t\le1/4$.
If $\delta=1/2-p\le1/50000$, Bernoulli's inequality gives
\[
 N(2^{1300})\ge R_t\ge Q_0(1-4D\delta^2)
 \ge\frac{17603180821337989277049}{17592186044416000000000}>1.
\]
Corollary~\ref{cor:dichotomy} completes the proof.
\end{proof}

The bound at $p=1/2$ concerns the concentration decay rate and does not
determine its exact value. For $0.49998\le p<1/2$, the certificate proves a
strict gap between the global and boundary H\"older exponents; it does not
locate the parameter at which this gap first occurs. The remaining question
is: \emph{for which $p\in(0,1/2)$ does $N(z)\le1$ hold for every $z$?}
By Corollary~\ref{cor:dichotomy}, precisely at those parameters the boundary
exponent is globally valid, with constant one.

\section{Unequal win and loss multipliers}
\label{sec:unequal-multipliers}

A natural extension is to multiply the target bet by $a>1$ after a
win and divide it by $b>1$ after a loss.  Thus $B_1=1$ and
\begin{equation}
\label{eq:ab-target-rule}
  B_{n+1}=
  \begin{cases}
    aB_n,&\xi_n=1,\\
    B_n/b,&\xi_n=-1,
  \end{cases}
\end{equation}
while the actual bet remains $C_n=\min\{B_n,W_{n-1}\}$.  Denote the
corresponding survival probability by $f_{a,b}(x,p)$.  The answer is
governed by the thresholds
\begin{equation}
\label{eq:ab-thresholds}
  d:=\frac{b}{b-1},
  \qquad
  a_*:=2-\frac1b,
  \qquad
  p_*:=\frac{\log b}{\log a+\log b}.
\end{equation}

\begin{theorem}[Two-multiplier phase diagram]
\label{thm:ab-phase}
Let $a,b>1$, $x\ge1$, and $0<p<1$.
\begin{enumerate}
\item If $a>a_*$, then
\[
  f_{a,b}(x,p)>0
  \quad\Longleftrightarrow\quad
  x>d\ \text{ and }\ p<p_*.
\]
In this region $0<f_{a,b}(x,p)<1$.

\item If $a=a_*$, then $f_{a,b}(x,p)=1$ for $x\ge d$.  For
$1\le x<d$,
\[
  f_{a,b}(x,p)>0
  \quad\Longleftrightarrow\quad p>p_*,
\]
and the positive survival probability is strictly less than one.

\item If $a<a_*$, then $f_{a,b}(x,p)>0$ for every $x\ge1$ and every
$0<p<1$.  More precisely, it equals one for $x\ge d$ and lies strictly
between zero and one for $1\le x<d$.
\end{enumerate}

At the endpoints, independently of the comparison between $a$ and
$a_*$, one has
\[
  f_{a,b}(x,0)=\one_{[d,\infty)}(x),
  \qquad
  f_{a,b}(x,1)=1.
\]
\end{theorem}

The original rule has $a=b=2$, hence $d=2$, $a_*=3/2$, and
$p_*=1/2$; it is therefore the first case of
Theorem~\ref{thm:ab-phase}. See Figure~\ref{fig:phase}.

\begin{figure}[htbp]
  \centering
  \includegraphics[width=\textwidth]{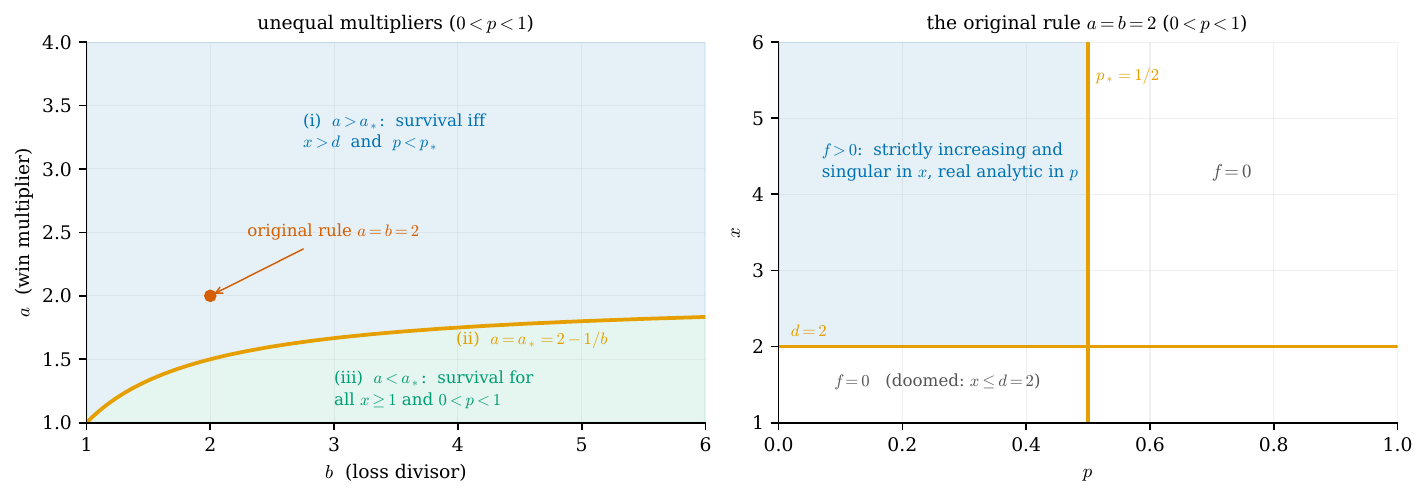}
  \caption{Phase diagrams for $0<p<1$ and $x\ge1$.
  Left: the three regimes of Theorem~\ref{thm:ab-phase}, separated by
  $a=2-1/b$. Above this curve, survival has positive probability exactly
  when $x>d$ and $p<p_*$. Below it, survival has positive probability
  for every initial fortune. On the curve, survival is certain for $x\ge d$,
  whereas for $x<d$ it has positive probability exactly when $p>p_*$.
  Right: the original rule $a=b=2$, for which $f>0$ exactly when
  $x>2$ and $p<1/2$. The deterministic endpoints are stated separately
  in Theorem~\ref{thm:ab-phase}.}
  \label{fig:phase}
\end{figure}

\begin{proof}
Let $\widehat W_n$ be the uncapped shadow wealth for
\eqref{eq:ab-target-rule}, and put
$\widehat Y_n:=\widehat W_n/B_{n+1}$.  Above the cap, this normalized
shadow ratio has the two transitions
\begin{equation}
\label{eq:ab-ratio-maps}
  F_a(y):=\frac{y+1}{a}\quad\text{on a win},
  \qquad
  G_b(y):=b(y-1)\quad\text{on a loss}.
\end{equation}
At a capped state $0<y\le1$, a win gives $2y/a$ and a loss gives zero.
The number $d$ is the fixed point of $G_b$.  Put
\begin{equation}
\label{eq:ab-kappa}
  \kappa:=(a-1)d-1
  =\frac{ab-2b+1}{b-1}.
\end{equation}
Thus $\kappa$ has the same sign as $a-a_*$.  Centering at $d$, with
$X_n:=\widehat Y_n-d$, turns \eqref{eq:ab-ratio-maps} into
\begin{equation}
\label{eq:ab-centered-maps}
  X\longmapsto\frac{X-\kappa}{a}\quad\text{on a win},
  \qquad
  X\longmapsto bX\quad\text{on a loss}.
\end{equation}
If
\[
  \mathcal S_{a,b}^{(n)}
  :=\sum_{j=1}^n\one_{\{\xi_j=1\}}B_j,
\]
then induction in \eqref{eq:ab-centered-maps} gives
\begin{equation}
\label{eq:ab-explicit}
  B_{n+1}X_n=x-d-\kappa\mathcal S_{a,b}^{(n)}.
\end{equation}

Suppose first that $\kappa>0$.  Once the capped ratio enters
$(0,d]$, it remains there until ruin.  Indeed, for $1<y\le d$,
$F_a(y)\le F_a(d)<d$ and $G_b(y)\le d$, while for $0<y\le1$ a capped win
gives $2y/a\le2/a<d$ and a loss ruins the gambler.  Entry is in fact
fatal almost surely, uniformly over the entry state.  Set
\[
  r_0:=\max\{F_a(d),2/a\}<d
\]
and choose $m$ so that
$G_b^m(r_0)=d+b^m(r_0-d)\le1$.  From every state in $(0,d]$, the outcome
word consisting of one win followed by $m+1$ losses causes ruin,
possibly before the word is completed.  This word has probability
$p(1-p)^{m+1}>0$, so its occurrence in successive disjoint blocks
shows that ruin follows entry almost surely.

At the first entry into $(0,d]$, the capped and shadow ratios still
agree, because the cap is inactive above $d$.  Consequently survival
is, up to a null event, the event that $X_n>0$ for every $n$.  Let
\[
  \mathcal S_{a,b}:=
  \sum_{j\ge1}\one_{\{\xi_j=1\}}B_j,
\]
with $+\infty$ allowed.  Since there are infinitely many wins almost
surely, every finite partial sum is strictly smaller than
$\mathcal S_{a,b}$ whenever the latter is finite.
Equation~\eqref{eq:ab-explicit} therefore yields, for $x>d$,
\begin{equation}
\label{eq:ab-perpetuity-representation}
  f_{a,b}(x,p)
  =\PP\left(\mathcal S_{a,b}\le\frac{x-d}{\kappa}\right),
\end{equation}
whereas $f_{a,b}(x,p)=0$ for $x\le d$.

The random walk $\log B_n$ has mean increment
\begin{equation}
\label{eq:ab-log-drift}
  \lambda:=p\log a-(1-p)\log b.
\end{equation}
If $\lambda<0$, the strong law makes $B_n$ decay exponentially, so
$\mathcal S_{a,b}<\infty$ almost surely.  If $\lambda>0$, then
$B_n$ grows exponentially and the infinitely many winning summands do
not tend to zero.  If $\lambda=0$, the centered nonconstant random walk
$\log B_n$ oscillates.  It has infinitely many upward crossings of
zero; each is a win whose summand $B_n$ is at least $1/a$.  Hence
$\mathcal S_{a,b}=\infty$ almost surely whenever $\lambda\ge0$.  This
is the finite-support instance of the classical perpetuity criterion
cited in the Introduction; see \cite{Vervaat1979,GoldieMaller2000}.
Since $\lambda<0$ is equivalent to $p<p_*$, it remains only to check
that every positive threshold in \eqref{eq:ab-perpetuity-representation}
has probability strictly between zero and one.  Conditional on the
first $M$ outcomes being losses,
\[
  \operatorname{Law}\!\left(
    \mathcal S_{a,b}\,\middle|\,
    \xi_1=\cdots=\xi_M=-1
  \right)
  =\operatorname{Law}\!\left(b^{-M}\mathcal S_{a,b}'\right),
\]
where the variable on the right is an independent copy.  Almost sure
finiteness therefore gives positive probability below any prescribed
positive threshold when $M$ is large.  Conversely, sufficiently many
initial wins contribute
$1+a+\cdots+a^{K-1}$, which exceeds that threshold for large $K$.
This proves part~(1).

Now suppose that $\kappa<0$.  Both maps in
\eqref{eq:ab-ratio-maps} preserve $[d,\infty)$, so this is a pathwise
safe region.  Moreover $a<a_*<2$, and the fixed point
$e:=1/(a-1)$ of $F_a$ satisfies $e>d$.  From any $x\in[1,d)$, a
sufficiently long string of wins reaches $[d,\infty)$: at $x=1$ the
first capped win gives $2/a=F_a(1)>1$, and thereafter the iterates of
$F_a$ increase to $e$.  This gives positive survival probability.  On
the other hand, from any $x<d$ a sufficiently long string of losses
enters the cap and the next loss causes ruin, so the survival
probability there is less than one.  This proves part~(3).

It remains to consider $\kappa=0$.  Now $a=a_*$ and both maps preserve
$[d,\infty)$, proving certain survival there.  During an uncapped
excursion with $1<y<d$, let
\[
  D:=d-y,
  \qquad \ell:=d-1.
\]
Until the cap is reached, a win sends $D$ to $D/a$ and a loss sends it
to $bD$.  Thus $\log D$ has increments
\[
  U=
  \begin{cases}
    -\log a,&\text{with probability }p,\\
    \log b,&\text{with probability }1-p,
  \end{cases}
\]
and mean $(1-p)\log b-p\log a$.  This mean is negative exactly when
$p>p_*$.  In that case choose $\theta>0$ small enough that
$\EE e^{\theta U}<1$.  If $D<\ell$, then
$\exp(\theta\sum_{j=1}^nU_j)$ is a supermartingale, and the maximal
inequality at the level $A:=\log(\ell/D)>0$ gives
\begin{equation}
\label{eq:ab-barrier-bound}
  \PP\left(\sup_{n\ge1}\sum_{j=1}^nU_j\ge A\right)
  \le e^{-\theta A}<1.
\end{equation}
There is therefore positive probability never to reach the cap.  From
$x=1$, one initial win leads to $2/a\in(1,d)$, after which the same
argument applies.

If $p\le p_*$, the random walk $\sum_{j=1}^nU_j$ has positive drift
or is centered and oscillating, so every uncapped excursion below $d$
hits the cap almost surely.  At each capped visit, an immediate loss
causes ruin with probability $1-p$.  To survive a capped episode,
every outcome until exit must be a win.  Since $2/a>1$, such an
episode either ends in ruin or exits after finitely many rounds into
$(1,2/a]\subset(1,d)$.  In the latter case another capped episode
follows almost surely.  By the strong Markov property, the probability
of surviving $N$ such visits is at most $p^N$, and hence ruin is almost
sure.  A finite string of losses has positive probability and causes
ruin from every $x<d$, so the positive survival probability for
$p>p_*$ is less than one.  This proves part~(2).

Finally, if $p=0$, the deterministic loss iterates satisfy
$G_b^n(x)=d+b^n(x-d)$ until the cap is reached, giving survival exactly
for $x\ge d$.  If $p=1$, losses never occur and bankruptcy is
impossible.  The endpoint assertions follow.
\end{proof}

\section*{Acknowledgments}

Yuval Peres is supported by the National Natural Science Foundation of China under Grant No.~W2531011. Shuo Qin is supported by the China Postdoctoral Science Foundation under Grant Numbers 2025M773086 and 2026T190815, and National Natural Science Foundation of China under Grant Number 12271284.

\bibliographystyle{amsplain}
\bibliography{references,references_additions}

\end{document}